\documentclass[11pt,  a4paper,  reqno,  dvipsnames,  svgnames,  table]{amsart}
\usepackage{amsmath,amssymb,amscd,amsthm,mathtools,mathrsfs}
\usepackage{dsfont}

\usepackage{bm}
\usepackage{graphicx}
\usepackage{tikz}
\usepackage{tikz-cd}
\usepackage{subcaption}
\usepackage{blkarray}
\usepackage{nicematrix}
\usepackage{booktabs}
\usepackage{url}
\usepackage{pifont}
\usepackage{xcolor}
\definecolor{GreenDef}{RGB}{27,132,5}
\definecolor{BlueDef}{RGB}{11,78,188}
\definecolor{RedDef}{RGB}{15,77,146}
\definecolor{Yblue}{RGB}{160,4,23}
\usepackage[top=2cm,bottom=2cm,left=1.9cm,right=1.9cm]{geometry}
\usepackage{titlesec}
\usepackage{titletoc}
\titleformat{\section}[hang]
{\scshape\small\color{RedDef}\filcenter}{\S\ \thesection.}{0.2em}{}
\titleformat{\subsection}[runin]
{\scshape\color{RedDef}}{\thesubsection.}{0.2em}{}[.]
\contentsmargin{2.55em}
\dottedcontents{section}[3.8em]{}{1.5em}{1pc}
\dottedcontents{subsection}[6.1em]{}{2.5em}{1pc}
\usepackage[breaklinks,colorlinks,pagebackref]{hyperref}
\definecolor{Red}{RGB}{160,4,23}
\definecolor{Green}{RGB}{27,132,5}
\definecolor{Blue}{RGB}{11,78,188}
\hypersetup{linkcolor=Red,citecolor=Green}
\usepackage[capitalise,nameinlink]{cleveref}
\crefformat{equation}{(#2#1#3)}
\usepackage[msc-links,lite,alphabetic]{amsrefs}
\usepackage{mathscinet}

\usepackage{ulem}
\def\MR#1{%
  \relax\ifhmode\unskip\spacefactor3000 \space\fi
  \href{http://www.ams.org/mathscinet-getitem?mr=#1}{MR#1}
}

\makeatletter
\renewcommand{\BibLabel}{%
  \Hy@raisedlink{\hyper@anchorstart{cite.\CurrentBib}\hyper@anchorend}%
  [\thebib]%
}
\makeatother

\usetikzlibrary{decorations}
\usepgflibrary{decorations.fractals}
\usepgflibrary{shapes.symbols}
\usepgflibrary{shapes.misc}
\usepackage{color}
\usepackage{fancybox}
\usepackage{fancyhdr}
\usepackage{dsfont}

\usepackage{xcolor}
\usepackage{fontenc,cleveref}
\usepackage{epigraph}

\usepackage{boxedminipage}
\usepackage{wasysym}
\usepackage{latexsym}
\usepackage[T1]{fontenc}
\usepackage[latin9]{inputenc}
\usepackage{framed}
\usepackage{verbatim}
\usepackage{xspace}
\usepackage{stmaryrd}
\definecolor{brightmaroon}{rgb}{0.76, 0.13, 0.28}

\usepackage{tikz}
\usetikzlibrary{arrows.meta,calc}
\usepackage{bm}
\usepackage{graphicx}
\usepackage{tikz}
\usepackage{tikz-cd}
\usepackage{subcaption}
\usepackage{blkarray}
\usepackage{nicematrix}
\usepackage{booktabs}
\usepackage{url}
\usepackage{pifont}
\usepackage{xcolor}
\definecolor{GreenDef}{RGB}{27,132,5}
\definecolor{BlueDef}{RGB}{11,78,188}
\definecolor{RedDef}{RGB}{15,77,146}
\definecolor{Yblue}{RGB}{160,4,23}
\usepackage[top=2cm,bottom=2cm,left=1.9cm,right=1.9cm]{geometry}
\usepackage{titlesec}
\usepackage{titletoc}
\titleformat{\section}[hang]
{\scshape\small\color{RedDef}\filcenter}{\S\ \thesection.}{0.2em}{}
\titleformat{\subsection}[runin]
{\scshape\color{RedDef}}{\thesubsection.}{0.2em}{}[.]
\contentsmargin{2.55em}
\dottedcontents{section}[3.8em]{}{1.5em}{1pc}
\dottedcontents{subsection}[6.1em]{}{2.5em}{1pc}
\usepackage[breaklinks,colorlinks,pagebackref]{hyperref}
\definecolor{Red}{RGB}{160,4,23}
\definecolor{Green}{RGB}{27,132,5}
\definecolor{Blue}{RGB}{11,78,188}
\hypersetup{linkcolor=Red,citecolor=Green}
\usepackage[capitalise,nameinlink]{cleveref}
\crefformat{equation}{(#2#1#3)}
\usepackage[msc-links,lite,alphabetic]{amsrefs}
\usepackage{mathscinet}

\def\MR#1{%
  \relax\ifhmode\unskip\spacefactor3000 \space\fi
  \href{http://www.ams.org/mathscinet-getitem?mr=#1}{MR#1}
}

\makeatletter
\renewcommand{\BibLabel}{%
  \Hy@raisedlink{\hyper@anchorstart{cite.\CurrentBib}\hyper@anchorend}%
  [\thebib]%
}
\makeatother
\theoremstyle{plain}
\newtheorem{theorem}{Theorem}
\newtheorem{lemma}[theorem]{Lemma}
\newtheorem{proposition}[theorem]{Proposition}

\theoremstyle{definition}

\newtheorem{remark}[theorem]{Remark}

\newcommand{\QQ}{\mathbb{Q}}

\newcommand{\ZZ}{\mathbb{Z}}

\title{An infinite family of trees with irreducible characteristic polynomials}

 \author[Saieed Akbari]{Saieed Akbari}
 \address{Department of Mathematical Science, Sharif University of Technology, Tehran, Iran}
\email{s akbari@sharif.edu}

\author{ Keivan Mallahi-Karai}
\address{Constructor University, School of Science, Campus Ring I, 28759 Bremen, Germany}
\email{kmallahika@constructor.university}
\date{}

\begin{document}
\maketitle

\begin{abstract}
Consider the tree obtained by attaching a leaf to the third vertex of a path with $n-1$ vertices. In this note, we prove that the characteristic polynomial of this tree is irreducible when $n$ belongs to certain arithmetic progressions modulo $30$. As a result there are infinitely many pairwise non-isomorphic trees with an irreducible the characteristic polynomial. Our proof combines several number-theoretic arguments with a result of Gross, Hironaka, and McMullen \cite{GHM} concerning the cyclotomic factors of the Coxeter polynomials associated with the diagrams $E_n$. This resolves affirmatively a conjecture of Akbari, Kumar, Mohar and Pragada.
\end{abstract}
\section{Introduction}

Let $G$ be a finite simple graph, and let $ \chi_G(t)=\det(tI-A(G))$ denote the characteristic polynomial of its adjacency matrix
$A(G)$. We say that $G$ is irreducible if $\chi_G(t)$ is irreducible as a polynomial in $\QQ[t]$. This is a strong arithmetic condition on the spectrum of $G$ as it implies, in particular, that all eigenvalues of $G$ are simple and form a single orbit under the action of the absolute Galois group of $\QQ$. Irreducibility of the characteristic polynomial is also closely connected with several structural and spectral properties of a graph. For instance, if $G$ is irreducible, then the automorphism group of $G$ is trivial.  For more properties of such graphs, see \cite{YLZH}.

Several methods for constructing graphs with irreducible characteristic polynomials have been developed. For example, Yu, Liu, Zhang, and Heng \cite{YLZH} give constructions based on Eisenstein's criterion and field extensions.  The problem becomes more challenging when attention is limited to trees. Akbari, Kumar, Mohar and Pragada conjectured that infinitely many irreducible trees exists. In fact, they even conjectured that for any odd $n \ge 7$, the graph obtained by by attaching a pendant edge to some vertex of the path of length $n$  is irreducible.
The question was raised during in the talk given by the first author in CanaDAM 2025. 

Although it seems plausible that many (perhaps even most) trees are irreducible, constructing an explicit infinite family of such trees is far from obviuos. In this paper, we provide such a family, thereby answering the question of Akbari, Kumar, Mohar and Pragada affirmatively.

Let us briefly mention some ideas of our approach. One of the main ingredients in our proof is a theorem of Gross, Hironaka, and McMullen \cite{GHM} concerning the cyclotomic factors of the Coxeter polynomials associated with the diagrams $E_n$. The characteristic polynomials of these graphs turn out to be closely related to their Coxeter polynomials. We combine this result with additional algebraic arguments relating the irreducibility of a polynomial $p(t)$ to that of $p(t^2)$. Together, these tools allow us to prove the irreducibility of the characteristic polynomials of infinitely many trees whose orders lie in certain arithmetic progressions.

\medskip

Let us continue with stating the main result of this note.
For $n \ge 3$ and $ 1 \le k \le n-1$, denote by $G_{n, k}$ the tree with the vertex set $\{ 1, \dots, n \}$ 
obtained from the path on the vertices $\{1,2,\dots,n-1 \}$ by attaching one extra leaf $v$ to the vertex $k$. We prove the following result.

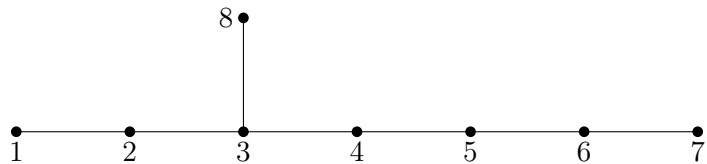
\begin{figure}
\begin{tikzpicture}

  \fill (0,0) circle (2pt) node[below] {1};
  \fill (1.5,0) circle (2pt) node[below] {2};
  \fill (3,0) circle (2pt) node[below] {3};
  \fill (4.5,0) circle (2pt) node[below] {4};
  \fill (6,0) circle (2pt) node[below] {5};
  \fill (7.5,0) circle (2pt) node[below] {6};
  \fill (9,0) circle (2pt) node[below] {7};
  \fill (3,1.5) circle (2pt) node[left] {8};

  \draw (0,0) -- (1.5,0) -- (3,0) -- (4.5,0) -- (6,0) -- (7.5,0) -- (9,0);
  \draw (3,0) -- (3,1.5);

\end{tikzpicture}
\caption{The tree \(G_{8,3}\).}
\label{fig:G83}
\end{figure}

\begin{theorem}\label{main}
For $n > 8$ satisfies $n$ is congruent to one of the numbers $2,10,16,20, 22, 26, 28$ mod $30$. 
Denote the the characteristic polynomial of $G_{n,3}$ by $F_n(x)$. Then $F_n(x)$ is irreducible over $\QQ$. In particular, there are infinitely many pairwise non-isomorphic trees with an irreducible  characteristic polynomial.
\end{theorem}

We will start by recalling some basic facts about Chebyshev polynomials that will turn out to be useful in describing the characteristic polynomial 
of $G_{n,3}$.

\section{General facts about Chebyshev polynomials of the second kind}
Recall that the Chebyshev polynomials of the second kind form a family of polynomials $U_m(x)$, indexed by $m=0,1,2,\ldots$, and defined by the initial conditions

\[
U_0(x)=1,\qquad U_1(x)=2x,
\]
and then inductively for $m\ge 2$ by
\[
U_m(x)=2x U_{m-1}(x)-U_{m-2}(x).
\]

For instance, $U_2(x)=4x^2-1$. 
We will use two other related forms of Chebyshev polynomials. It is not difficult to see that for each $m \ge 1$, we have
 $$U_m(\cos \theta)=  \frac{\sin(m+1)\theta}{\sin \theta}.$$
Writing $z= \cos \theta+ i \sin \theta$, we can express this property as
\[  U_m  \left( \frac{z+z^{-1}}{2} \right) = \frac{z^{m+1}- z^{- (m+1)} }{z- z^{-1}}. \]

This representation will be used later. We now reformulate the recurrence relation in a way that will allow us to linearize several expressions involving Chebyshev polynomials that arise in the sequel.

\begin{proposition}\label{producttosum} For all $m \ge 1$ we have
\[ t U_m \left(\frac {t}{2}\right) = U_{m+1} \left(\frac {t}{2}\right) + U_{m-1} \left(\frac {t}{2}\right),
\]
\end{proposition}

\begin{proof}
This is precisely the recursive formula written for $m+1$ and $x= t/2$. 
\end{proof}

\begin{remark}\label{special}
We will make repeated use of this identity. For instance, applying it successively yields
\[ t^2 U_m \left(\frac {t}{2}\right) = t  \left(  U_{m+1} \left(\frac {t}{2}\right) + U_{m-1} \left(\frac {t}{2}\right) \right)=
 U_{m+2} \left(\frac {t}{2}\right) + 2 U_{m} \left(\frac {t}{2}\right)+  U_{m-2} \left(\frac {t}{2}\right).
 \]
\end{remark}

\section{Characteristic polynomials}
\begin{proposition}\label{char:path}
The characteristic polynomial of the path $P_m$ on $m$ vertices is given by 
\[
\chi_{P_m}(t)=U_m\!\left(\frac {t}{2}\right).
\]
\end{proposition}

\begin{proof}
This is well known. See, for instance \cite[p.47]{CDRPS}.
\end{proof}

\begin{proposition}\label{char:tree}
Let $G:=G_{n, 3}$ as above. Thus the characteristic polynomial of $G$ is given by
\[
\chi_G(t)
= U_n(t/2)- U_{n-4}(t/2)-  U_{n-6}(t/2). 
\]

\end{proposition}

\begin{proof}[Proof of Proposition \ref{char:tree}]
We use the recursive formula \cite[p.17]{Biggs} for the the characteristic polynomial 
\[
\chi_G(t)=t\chi_{G-\{ v\} }(t)-\chi_{G- \{v,  k\} }(t).
\]

Note that $G-\{ v\} \simeq P_{n-1}$ and $G- \{v,  k\} \simeq P_{k-1} \cup P_{n-k-1}$. 
Hence, 
\[
\chi_G(t) = t\chi_{P_{n-1}}(t) - \chi_{P_{k-1}}(t)\chi_{P_{n-k-1}}(t).
\]

Using Proposition \ref{char:path} we deduce
\[
\chi_G(t) = t U_{n-1} \left(\frac{t}{2} \right) - U_{k-1} \left(\frac {t}{2}\right) U_{n-k-1} \left(\frac {t}{2}\right).
\]

Using Proposition \ref{producttosum} and Remark \ref{special}, we can rewrite this as 
\begin{equation}
\begin{split}      
 \chi_{G_{n,3}}(t) &= t U_{n-1}(t/2)- (t^2-1) U_{n-4}(t/2) \\
&= U_n(t/2)+ U_{n-2}(t/2)- (U_{n-2}(t/2)+2 U_{n-4}(t/2)+ U_{n-6}(t/2))+ U_{n-4}(t/2) \\
&=  U_n(t/2)- U_{n-4}(t/2)-  U_{n-6}(t/2). 
\end{split}
\end{equation}
\end{proof}

Theorem \ref{main} will now follow from the following theorem. 

\begin{theorem}\label{irreducibility}
For $n > 8$ satisfies $n$ is congruent to one of the numbers $2,10,16,20, 22, 26, 28$ mod $30$. 
Consider the polynomials of degree $n$ given by
\[ 
F_n(x):=U_n(x)-U_{n-4}(x)-U_{n-6}(x)
\]
Then $F_n(x)$ is irreducible over $\QQ$.
\end{theorem}

A major step toward proving our theorem is already contained in the work of Gross, Hironaka, and McMullen \cite{GHM}. Nevertheless, some additional arguments are required to deduce our result from theirs. In their paper, Gross, Hironaka, and McMullen study the Coxeter polynomials associated with the so-called $E_n$ diagrams. The precise definition of these diagrams is
not relevant for us. They are explicitly given by
\[ E_n(x)= \frac{x^{n-2}Q(x)+R(x)}{x-1} \]
where
$Q(x)= x^3-x-1$ and $R(x)=x^3+x^2-1$; see \cite[p.1035]{GHM}. In particular, they prove that

\begin{theorem}[Gross-Hironaka-McMullen, Corollary 1.4]\label{GHM}
If $n>8$ is congruent to one of $2,10,16,20, 22, 26, 28$ mod $30$ then $E_n(x)$ is irreducible over $\QQ$. 
\end{theorem}

The connection between  these two families of polynomials is stated in the next lemma. 

\begin{lemma}\label{EF}
For all values of $n \ge 2$ we have
\[ E_n(z^2)= z^{n}F_n \left(\frac{z+z^{-1}}{2}\right) \]
\end{lemma}

\begin{proof} This is a simple verification that we will include the details for completeness. 
Using the properties of Chebyshev polynomials discussed before, we can write 
\begin{equation}
\begin{split}      
F_n \left(\frac{z+z^{-1}}{2}\right) &= U_n \left(\frac{z+z^{-1}}{2}\right)- U_{n-4} \left(\frac{z+z^{-1}}{2}\right) - U_{n-6} \left(\frac{z+z^{-1}}{2}\right) \\
&=\frac{1}{z-z^{-1}} \Bigl[(z^{n+1}-z^{-(n+1)})-(z^{n-3}-z^{-(n-3)})-(z^{n-5}-z^{-(n-5)})\Bigr].
\end{split}
\end{equation}

Multiply by $(z-z^{-1})z^{n+5}$ gives
 
\begin{equation}
\begin{split}      
(z-z^{-1})z^{n+5}F_n \left(\frac{z+z^{-1}}{2}\right) &=z^{2n+6}-z^4-z^{2n+2}+z^8-z^{2n}+z^{10}\\
&=z^4\Bigl(z^{2n+2}-z^{2n-2}-z^{2n-4}+z^6+z^4-1\Bigr) \\
&= z^4   \left( (z^2)^{(n-2)} (z^6 -z^2-1)+ z^6 +z^4- 1 \right) \\
 &= z^4 ( (z^2)^{(n-2)} Q(z^2)+ R(z^2)). 
\end{split}
\end{equation}
 
Dividing by $z^4$ we deduce
\[
 (z^2-1) z^n F_n \left(\frac{z+z^{-1}}{2}\right)=  (z^2)^{(n-2)} Q(z^2)+ R(z^2) = (z^2-1) E_n (z^2). \]
which proves the claim.

\end{proof}

Note that in general the irreducibility of $f(z) \in \QQ[z]$ does not imply the irreducibility of $f(z^2)$; take, for instance, $f(z)=z-4$. However, we will show that $E(z^2)$ is also irreducible over $\QQ$.
\begin{theorem}\label{ez2}
If $n>8$ is congruent to one of $2,10,16,20, 22, 26, 28$ mod $30$ then $E_n(z^2)$ is irreducible over $\QQ$. 
\end{theorem}

In order to prove this theorem we need the following simple facts from field theory. 

\begin{lemma}\label{dichotomy}
Suppose $E(z) \in \ZZ[z]$ is a monic irreducible polynomial of even degree. Then either $E(z^2)$ is irreducible over $\ZZ[z]$ or there exists an irreducible polynomial $P(z) \in \ZZ[z]$ such that 
$E(z^2)= P(z) P(-z)$. 
\end{lemma}

\begin{proof}
Set
\[
F(z):=E(z^2).
\]
Then $\deg F=2n$, where $n = \deg E$. 
Let $\alpha \in \overline{\QQ}$ to a root of $E$ and pick $\beta \in \overline{\QQ}$ such that $\beta^2=\alpha.$ Note that since $\alpha$ and
$\beta$ are roots of monic polynomials $E(z)$ and $E(z^2)$ and hence are both algebraic integers. Let $P(z)$ be the minimal polynomial of $\beta$ over $\QQ$. Since $\beta$ is an algebraic integer, it follows from  \cite[Chapter 2. Theorem 1]{Marcus} that $P(z) \in \ZZ[z]$. Note that $Q(z):= P(-z)$ is the minimal polynomial of $-\beta$ and since $F(\pm \beta)= E(\alpha)=0$, we deduce 
\begin{equation}\label{divis}
P(z)|F(z) \text{ \, and \, } Q(z)|F(z). 
\end{equation}

Since $E(z)$ is the minimal polynomial of $\alpha$ over $\QQ$ and has degree $n$, it follows that
\[
[\mathbb{Q}(\alpha):\mathbb{Q}]=n.
\]
Since  $\beta$ is a root of $x^2-\alpha\in \mathbb{Q}(\alpha)[x] $ it follows that the extension $ \QQ(\beta)/\QQ (\alpha)$ has degree at most $2$. 
\
From here we deduce that
\[
\deg P=[\QQ(\beta):\QQ] = [\QQ(\beta):\QQ(\alpha)] [\QQ(\alpha):\QQ] \in\{n,2n\}.
\]

We consider two cases:

\noindent
{\it Case 1. } Suppose $\deg P(z)=2n$. Since $P(z)|F(z)$ by \eqref{divis}, we have  $$E(z^2)= F(z)= P(z)$$ is irreducible.  

\medskip 
\noindent
{\it Case 2}. Now suppose that $\deg P(z)=n$. Using  \eqref{divis} we can write $F(z)=P(z) S(z)$ for some monic polynomial $S(z)$ of degree $n$.

Note that $F(z)=E(z^2)$ is invariant under $ z \mapsto -z$. Hence we have
$F(z)= P(-z) S(-z)$. If $P(z)=P(-z)$, then a coefficient comparison shows  that odd-degree coefficients of $P(z)$ are zero and $P(z)$ is indeed a polynomial in $z^2$: $P(z)=R(z^2)$. This shows that 
$$E(z^2)= R(z^2) S(z),$$ 
which implies that $S(-z)=S(z)$, that is, $S(z)$ is also a polynomial in $z^2$. Write $S(z)=T(z^2)$. From here we obtain the decomposition $E(z)= R(z) T(z)$ which contradicts the irreducibility of $E(z)$. So, $P(z)$ and $P(-z)$ are two distinct irreducible polynomials, both dividing $F(z)$. Since $\deg F(x)= \deg P(z)+ \deg P(-z)$, we deduce $E(z^2)= P(z)P(-z)$. This finishes the proof in the second case.
\end{proof}

\begin{proof}[Proof of Theorem \ref{ez2}]
Since $E_n(z)$ has degree $n$ and condition on $n$ module $30$, forces $n$ to be even, then using Lemma \ref{dichotomy}, we need to rule out the possibility that $E(z^2)= P(z) P(-z)$ for some integer monic polynomial $P(z)$ of degree $n=2m$. Separating the odd and even degree terms, write $P(z)= A(z^2)+ z B(z^2)$ for 
integer polynomials $A(z)$ and $B(z)$. It thus follows that 
\[ E_n(z^2)= P(z)P(-z)= (A(z^2)+ z B(z^2))(  A(z^2)- z B(z^2))= A(z^2)^2- z^2 B(z^2)^2.\]
Writing $y=z^2$ we must have $$E(y)= A(y)^2 - y B(y)^2.$$ 

Assume $n >5$. We write
\[ E_n(y) = \frac{y^{n+1}-y^{n-1}- y^{n-2} + y^3+ y^2-1  }{y-1}= \frac{y^{n+1}-1}{y-1} -  \frac{y^{n-1}-1}{y-1} -  \frac{y^{n-2}-1}{y-1} +  \frac{y^3-1}{y-1} +  \frac{y^2-1}{y-1}. \]
Note that each summand $ \frac{y^r-1}{y-1}$ is a sum of powers $y^j$ with $0 \le j \le r-1$. Keeping track of the cancellation of terms we can easily see that the coefficients of 
$y^j$ in the sum is given by 
\[ c_j= \begin{cases} 1 & \textrm{if } \, j \in \{ n,n-1,1,0 \}\\
0 & \textrm{if } \, j \in \{ n-2,2 \}   \\
1 & \textrm{if } \, 3 \le j \le n-3 \\
\end{cases}    \]
Writing $n=2m$, it is clear that $\deg A(y)=m$ and $\deg B(y) \le m-1$. Note that since $P(z)= A(z^2)+ z B(z^2)$ and $P(z)$ is monic, hence $A(y)$ is monic. We expand both polynomials as 
\[ A(y)= y^m + a_{m-1}y^{m-1} + \cdots + a_1 y +a_0, \quad B(y)= b_{m-1} y^{m-1} + \cdots + b_1 y + b_0, \]

where all coefficients $a_i, b_i$ are integers. We will now derive a contradiction by comparing these coefficients. First note that since $c_{2m-1}=1$, hence the coefficients of $y^{2m-1}$ in $ A(y)^2 - y B(y)^2$ is $1$. This implies that 
\[ 2a_{m-1} - b_{m-1}^2=1 \hspace{2mm} \Rightarrow \hspace{3mm} b_{m-1} \equiv 1 \pmod{ 2 }
\hspace{2mm} \Rightarrow \hspace{3mm} b_{m-1}^2 \equiv 1 \pmod{ 8 } \hspace{2mm} \Rightarrow \hspace{3mm}  a_{m-1} \equiv 1 \pmod{ 4 }. \]
In particular, $a_{m-1}$ is odd. Now, since the coefficient of $y^{2m-2}$ in $E_n(y)$ is zero, we have
\[ 0= a_{m-1}^2 +2a_{m-2} - 2b_m b_{m-1} \equiv 1 \pmod{ 2 }, \]
which is a contradiction. This proves that $E(z^2)$ is irreducible. 
\end{proof}

\begin{proof}[Proof of Theorem \ref{main}]
By Theorem \ref{ez2}, we know that $E_n(z^2)$ is irreducible over $\QQ$. We now prove that $F_n$ is irreducible over $\QQ$. To see this suppose that 
$F_n(z)=g(z)h(z)$ with $g,h\in\mathbb{Q}[z]$ of degree at least $1$. 
Let $\deg g=r$, $\deg h=s$ so $r+s=n$.

Consider the polynomials 
\[
G(z):=z^r g \left(\frac{z+z^{-1}}{2}\right), \quad  H(z):=z^s h \left(\frac{z+z^{-1}}{2}\right). 
\]

Then $G$ and $H$ have degree $2r$ and $2s$, respectively. Moreover, we have 
\[
E_n(z^2)=  z^n F_n \left(\frac{z+z^{-1}}{2}\right)=G(z)H(z).
\]
which is a contradiction.  This contradicts Theorem \ref{ez2}.

\end{proof}

\bibliographystyle{amsplain}   
\bibliography{irr}

@article {YLZH,
    AUTHOR = {Yu, Qian and Liu, Fenjin and Zhang, Hao and Heng, Ziling},
     TITLE = {Note on graphs with irreducible characteristic polynomials},
   JOURNAL = {Linear Algebra Appl.},
  FJOURNAL = {Linear Algebra and its Applications},
    VOLUME = {629},
      YEAR = {2021},
     PAGES = {72--86},
      ISSN = {0024-3795,1873-1856},
   MRCLASS = {05C50 (05C31 11R09 12F05)},
  MRNUMBER = {4293741},
       DOI = {10.1016/j.laa.2021.07.013},
       URL = {https://doi.org/10.1016/j.laa.2021.07.013},
}

@book {CDRPS,
    AUTHOR = {Cvetkovic, Dragovs and Rowlinson, Peter and Simic,
              Slobodan},
     TITLE = {An introduction to the theory of graph spectra},
    SERIES = {London Mathematical Society Student Texts},
    VOLUME = {75},
 PUBLISHER = {Cambridge University Press, Cambridge},
      YEAR = {2010},
     PAGES = {xii+364},
      ISBN = {978-0-521-13408-8},
   MRCLASS = {05-02 (05C50)},
  MRNUMBER = {2571608},
MRREVIEWER = {Ligong\ Wang},
}

@book {Marcus,
    AUTHOR = {Marcus, Daniel A.},
     TITLE = {Number fields},
    SERIES = {Universitext},
   EDITION = {Second},
      NOTE = {With a foreword by Barry Mazur},
 PUBLISHER = {Springer, Cham},
      YEAR = {2018},
     PAGES = {xviii+203},
      ISBN = {978-3-319-90232-6; 978-3-319-90233-3},
   MRCLASS = {11-01 (11Rxx 11Txx 12-01)},
  MRNUMBER = {3822326},
       DOI = {10.1007/978-3-319-90233-3},
       URL = {https://doi.org/10.1007/978-3-319-90233-3},
}

@article {GHM,
    AUTHOR = {Gross, Benedict H. and Hironaka, Eriko and McMullen, Curtis
              T.},
     TITLE = {Cyclotomic factors of {C}oxeter polynomials},
   JOURNAL = {J. Number Theory},
  FJOURNAL = {Journal of Number Theory},
    VOLUME = {129},
      YEAR = {2009},
    NUMBER = {5},
     PAGES = {1034--1043},
      ISSN = {0022-314X,1096-1658},
   MRCLASS = {11C08 (12E10 20F55)},
  MRNUMBER = {2516970},
MRREVIEWER = {Michael\ E.\ Zieve},
       DOI = {10.1016/j.jnt.2008.09.021},
       URL = {https://doi.org/10.1016/j.jnt.2008.09.021},
}

@book {Biggs,
    AUTHOR = {Biggs, Norman},
     TITLE = {Algebraic graph theory},
    SERIES = {Cambridge Mathematical Library},
   EDITION = {Second},
 PUBLISHER = {Cambridge University Press, Cambridge},
      YEAR = {1993},
     PAGES = {viii+205},
      ISBN = {0-521-45897-8},
   MRCLASS = {05C50},
  MRNUMBER = {1271140},
MRREVIEWER = {Robin\ J.\ Wilson},
}

\end{document}